\documentclass[12pt]{amsart}
\usepackage[active]{srcltx}
\usepackage{a4wide,amssymb}
\usepackage[all, arrow, matrix, curve, cmtip]{xy}

\newtheorem{theorem}{Theorem}[section]
\newtheorem{lemma}[theorem]{Lemma}

\newtheorem{corollary}[theorem]{Corollary}
\theoremstyle{definition}
\newtheorem{definition}[theorem]{Definition}

\theoremstyle{remark}

\newtheorem{example}[theorem]{Example}

\numberwithin{equation}{section}

\newcommand {\D}  {{\mathcal D}}
\newcommand {\E}  {{\mathcal E}}

\newcommand {\N}  {{\mathcal N}}

\newcommand {\R}  {{\mathcal R}}

\newcommand {\FF}  {{\mathbb F}}

\newcommand {\QQ}  {{\mathbb Q}}

\newcommand {\ZZ}  {{\mathbb Z}}

\newenvironment{romanlist}
  {%
   \setlength{\topsep}{0pt}%
   \vspace{-\parskip}%
   \begin{enumerate}%
     \setlength{\parsep}{0pt}
     \setlength{\parskip}{0pt}%
  }%
  {\end{enumerate}%
   \vspace{-\parskip}}

\DeclareMathOperator{\lcm}{lcm}
\DeclareMathOperator{\Res}{Res}

\title{Number systems and the Chinese Remainder Theorem}

\author[C.~E.~van de Woestijne]{Christiaan~E.~van de Woestijne}
\address{Chair of Mathematics and Statistics, University of Leoben, A-8700 Leoben, AUSTRIA}
\email{c.vandewoestijne@unileoben.ac.at}
\dedicatory{Dedicated to Professor Attila Peth\H o on occasion of his 60th birthday}
\thanks{This research was supported by the Austrian Science Foundation (FWF),
project S9611, which is part of the national research network FWF-S96
``Analytic combinatorics and probabilistic number theory''.}

\date{\today}



\begin{document}

\begin{abstract}
  A well-known generalisation of positional numeration systems is the case
  where the base is the residue class of $x$ modulo a given polynomial $f(x)$
  with coefficients in (for example) the integers, and where we try to
  construct finite expansions for all residue classes modulo $f(x)$, using a
  suitably chosen digit set. We give precise conditions under which direct or
  fibred products of two such polynomial number systems are again of the same
  form. The main tool is a general form of the Chinese Remainder Theorem. We
  give applications to simultaneous number systems in the integers.
\end{abstract}

\maketitle

\begin{section}{Introduction}

Digit systems are a generalisation of the everyday positional numeration
systems, such as the decimal or binary. The most general definition in an
Abelian context is as follows.

\begin{definition}
  A \emph{digit system} in an Abelian group $V$ is a triple $(V,\phi,\D)$, where
  $\phi:V\rightarrow V$ a homomorphism with finite cokernel, and $\D\subset V$
  a finite subset that covers all cosets of $V/\phi(V)$. If there are $d_1$ and
  $d_2\in \D$ such that $d_1\equiv d_2\pmod{\phi(V)}$, we call $\D$ and also the
  digit system $(V,\phi,\D)$ \emph{redundant}; if $\D$ exactly represents $V$
  modulo $\phi(V)$, both it and $(V,\phi,\D)$ are \emph{irredundant}.

  The digit system $(V,\phi,\D)$ has the Finite Expansion Property if every
  element $v\in V$ can be written in the form 
  $$
    v=\sum_{i=0}^\ell \phi^i(d_i)
  $$
  for certain $d_i\in \D$. In this case, we call $(V,\phi,\D)$ a \emph{number
  system}, and $\D$ is called a \emph{valid digit set} for $(V,\phi)$.
\end{definition}

Note that we can expect \emph{unique} expansions in a digit system only when
the digit set is irredundant.

The generality of this definition will be needed only occasionally in the
paper. Mostly, we restrict ourselves to polynomial digit systems, which are
defined as follows.

\begin{definition}
  A \emph{polynomial digit system} is a digit system of the form
  $(\E[x]/(f),X,\D)$, where $\E$ is a commutative ring, $f\in\E[x]$ a
  nonconstant polynomial such that neither the leading nor the constant
  coefficient is a zero divisor in $\E$, and $X$ is the residue class of $x$
  modulo $f$.
\end{definition}

Note that in this case, the digit set $\D$ consists of polynomials such that
their constant coefficients cover the cosets of $\E/(f(0))$. For examples, see
\cite{ScheicherSurerThuswaldnerVanDeWoestijne:11,VanDeWoestijne:09}.

The following definition is so natural as to appear just a tautology.

\begin{definition}
  The \emph{direct product} of digit systems $(V_1,\phi_1,\D_1)$ and
  $(V_2,\phi_2,\D_2)$ is
  $$
    (V_1 \times V_2, \phi_1\times\phi_2, \D_1\times\D_2),
  $$
  where the first $\times$ denotes the direct product of groups.
\end{definition}

The main question in this paper will be whether the direct product of two
\emph{polynomial digit systems} is again a polynomial digit system. That is,
given $(\E[x]/(f_1),X,\N_1)$ and $(\E[x]/(f_2),X,\N_2)$, we consider
the question whether an isomorphism
\begin{equation} \label{EqDirectProduct}
  (\; \E[x]/(f),\, X,\, \N \;)
  \cong 
  \left( \; \E[x]/(f_1) \times \E[x]/(f_2),\, X\times X,\, \N_1\times\N_2
  \;\right) 
\end{equation}
holds for some $f$ and $\N$. The goal is to reduce the study of more
complicated number systems to systems modulo polynomials of lower degree,
because an isomorphism preserves the Periodic Representation and Finite
Expansion properties, if present.

Using a suitable generalisation of the Chinese Remainder Theorem, we arrive at
the following conclusions.

\begin{theorem}
  Let $\E$ be a PID, and let $f_1$, $f_2\in\E[x]$ be coprime. The map
  \begin{equation} \label{EqPsi} 
    \E[x]/(f_1f_2) \stackrel{\psi}{\longrightarrow}\E[x]/(f_1)\times\E[x]/(f_2)
  \end{equation}
  sending $a$ to $(a\mod{f_1},a\mod{f_2})$ is injective. It is surjective if
  and only if the ideal $(f_1,f_2)$ is the unit ideal of $\E[x]$.
\end{theorem}

\begin{corollary}
  The isomorphism \eqref{EqDirectProduct} holds with $f=f_1f_2$ if and only if
  $(f_1,f_2)$ is the unit ideal, and we have $\N=\psi^{-1}(\N_1\times \N_2)$.
\end{corollary}

When the leading coefficients of $f_1$ and $f_2$ are coprime in $\E$, then the
condition $(f_1,f_2)=(1)$ may be checked by checking that $\Res(f_1,f_2)$ is a
unit in $\E$ (see Lemma~\ref{LemRes} below).

Even if the isomorphism of the underlying groups in \eqref{EqDirectProduct}
does not hold, it is still possible in certain cases to embed the left side of
\eqref{EqDirectProduct} into the right hand side as a \emph{sub-number system}.
For this to hold, however, there are rather heavy restrictions on both $f_1$
and $f_2$ and the digits $\N_1$ and $\N_2$; in particular, in most cases $\N_1$
and $\N_2$ cannot contain $0$. The exact details will be given in Theorems
\ref{ThmCRTNumberSystem} and \ref{ThmMaximalPeriod}.

Our construction will also give a clear characterisation of \emph{simultaneous
number systems}, as defined in \cite{IndlekoferKataiRacsko:1992}. Among others,
we obtain an easy proof of the following, where we use a recent theorem on
products of linear polynomials independently due to Kane \cite{Kane:2006} and
Peth\H o \cite{Petho:2006}. Recall that a polynomial $f\in\ZZ[x]$ is called a
\emph{CNS polynomial} if the polynomial digit system $(\ZZ[x]/(f),\,X,\,
\{0,1,\ldots,|f(0)|-1\})$ has the Finite Expansion Property (for more on this
concept, see \cite[Section 3.1]{BaratBertheLiardetThuswaldner:2006}). 

\begin{theorem}
  Let $N_1,\ldots,N_k$ be distinct integers with $N_j\le -2$ for all $j$. If
  $\prod_{j=1}^k (x-N_j)$ is a CNS polynomial, then every integer $a$ has
  a unique simultaneous expansion of the form
  $$
    a = \sum_{i=0}^\ell d_i N_j^i \qquad (j=1,\ldots,k),
  $$
  where the $d_i$ are in $\{0,1,\ldots,|N_1\cdots N_k|-1\}$ and are the same
  for all $k$ bases $N_j$. In particular, the conclusion holds whenever $k\le
  4$.
\end{theorem}

In a more general context, we reduce the existence of such simultaneous number
systems to the algebraic-geometric problem of finding sets of polynomials with
coefficients in a given ring that pairwise have unit resultant. For example, it
seems to be unclear if there exists such a set of infinite cardinality.

All results in this paper are proved under the assumption that the ground ring
$\E$ is a principal ideal domain, unless stated otherwise.

\end{section}

\begin{section}{Algebraic background}

\begin{subsection}{The Chinese Remainder Theorem}

We will need a rather more general form of the CRT than usual. Recall that if
rings $A_1$ and $A_2$ map via homomorphisms $\pi_1$ and $\pi_2$ to a third ring
$B$, then the \emph{fibred product} of $A_1$ and $A_2$ over $B$ is defined as
$$
  A_1 \times_B A_2 = \{ (a_1,a_2) \in A_1\times A_2 :
                        \pi_1(a_1) = \pi_2(a_2) \}.
$$
It is a subring of the direct product $A_1 \times A_2$.

\begin{theorem}[\bf Chinese Remainder Theorem]\label{ThmCRT}
  Let $R$ be a commutative ring, with ideals $I$ and $J$. Then the map
  $\psi:R\rightarrow R/I\times R/J$, defined as
  $\psi(a)=(a\bmod{I},a\bmod{J})$, induces an isomorphism
  $$
    R/(I\cap J) \cong R/I \times_{R/(I+J)} R/J.
  $$
\end{theorem}

\begin{proof}
  Clearly, the kernel of $\psi$ is $I\cap J$. Thus, it remains to prove that
  $(a,a')$ is in the image of $\psi$ if and only if
  \begin{equation} \label{EqMod}
    a\bmod{I+J}= a'\bmod{I+J}.
  \end{equation}
  One inclusion is clear: given $a\in R$, clearly $(a\bmod{I})\bmod{I+J} =
  (a\bmod{J})\bmod{I+J}$. Now let $a\in R/I$ and $a'\in R/J$ satisfy
  \eqref{EqMod}. This means that $a+I+J=a'+I+J$, so there exists $u\in I$ and
  $v\in J$ with $a+u=a'+v$. But then $\psi(a+u)=(a,a')$, as desired, and the
  proof is done.
\end{proof}

In the following, we will follow established usage in calling elements of a
factorial ring that have trivial greatest common divisor \emph{coprime},
although this should actually mean that these elements together generate the
unit ideal. The notation $(a,b)$ denotes the \emph{ideal} generated by elements
$a$ and $b$, whereas $\gcd(a,b)$ denotes their gcd.

\begin{corollary} \label{CorCRT}
  Let $\E$ be a factorial ring, and let $f_1,f_2$ be in $\E[x]$ coprime. Then
  we have an isomorphism
  $$
    \E[x]/(f_1f_2) \cong \E[x]/(f_1) \times_{\E[x]/(f_1,f_2)} \E[x]/(f_2).
  $$
  In particular, given $a_1\in\E[x]/(f_1)$ and $a_2\in\E[x]/(f_2)$, there
  exists $a\in\E[x]/(f_1f_2)$ with $a\equiv a_i\bmod{f_i}$ ($i=1,2$) if and
  only if
  $$
    a_1\equiv a_2\bmod{(f_1,f_2)}.
  $$
\end{corollary}

\begin{proof}
  We apply the Theorem to the principal ideals $(f_1)$ and $(f_2)$. Then
  because $\E$ is factorial, also $\E[x]$ is factorial, and we have $(f_1)\cap
  (f_2)=(\lcm(f_1,f_2))$; and because $f_1$ and $f_2$ are coprime, we have
  $\lcm(f_1,f_2)=f_1f_2$.
\end{proof}

\begin{corollary} \label{CorCRTField}
  Let $K$ be a field, and let $f_1$ and $f_2$ in $K[x]$ be coprime. Then we
  have
  $$
    K[x]/(f_1f_2) \cong K[x]/(f_1) \times K[x]/(f_2).
  $$
\end{corollary}

\begin{proof}
  Over $K$, we have $(f_1,f_2)=(1)$ whenever $f_1$ and $f_2$ are coprime.
  Thus, the fibred product is over the \emph{zero ring} $K[x]/(1)$, and
  therefore equal to the direct (Cartesian) product.
\end{proof}

The conclusion of Corollary \ref{CorCRTField}, and hence the conventional
Chinese Remainder Theorem, is also true over a factorial ring $\E$ if, and only
if, the ideal $(f_1,f_2)$ is the unit ideal of $\E[x]$, so that
$\E[x]/(f_1,f_2)$ is the zero ring. The next result describes this situation
and extends it to products of more than $2$ factors.

\begin{corollary} \label{Cor111}
  Let $\E$ be a factorial ring, let $f_1,\ldots,f_k\in\E[x]$ be pairwise
  coprime, and let $R_i=\E[x]/(f_i)$ for $i=1,\ldots,k$. Define
  $$
    \psi:\E[x]/(f_1\cdots f_k) \rightarrow R_1 \times \cdots \times R_k
  $$
  by $a\mapsto (a\bmod f_1, \ldots, a \bmod f_k)$, and let $W$ be the image of
  $\psi$. Then
  $$
    W = \{ (a_1,\ldots,a_k) \mid a_i\equiv a_j \!\!\pmod{(f_i,f_j)} \text{ for }
      1\le i < j \le k \}.
  $$
  In particular, $\psi$ is surjective if and only if $(f_i,f_j)=(1)$ whenever
  $i\ne j$.
\end{corollary}

\begin{proof}
  By induction, where we use Corollary \ref{CorCRT} and the inclusions
  $(f_1\cdots f_{k-1},f_k)\subseteq (f_i,f_k)$ for $1\le i\le k-1$.
\end{proof}

\begin{corollary} \label{CorInterp}
  Assume the notations of Corollary \ref{Cor111}. Then an ordered tuple
  $(a_1,\ldots,a_k)\in R_1\times \cdots \times R_k$ is integrally interpolable
  by a polynomial in $\E[x]$ if and only if it is in $W$.
\end{corollary}

Note that any tuple $(a_1,\ldots,a_k)$ is interpolable by a polynomial over
the quotient field of $\E$; the question is whether this polynomial has
integral coefficients.

\begin{proof}
  An element $a\in\E[x]$ interpolates $(a_1,\ldots,a_k)$ whenever 
  $a\equiv a_i\pmod{f_i}$ for $i=1,\ldots,k$, that is, whenever
  $\psi(a)=(a_1,\ldots,a_k)$.
\end{proof}

\end{subsection}

\begin{subsection}{Strong Gr\"obner bases}

In order for Corollary \ref{CorCRT} to be useful, we will need a description in
some detail of the rings $\E[x]/(f_1,f_2)$ for polynomials $f_1,f_2\in\E[x]$.
Now the ring $\E[x]$ need not be a PID, even if $\E$ is, and in fact the
structure of generating sets of ideals in $\E[x]$ can be rather complicated.
For the case where $\E$ is a PID, a normal form for ideals in $\E[x]$ (nowadays
called \emph{strong Gr\"obner basis}) which at least permits to describe the
additive structure of the quotient ring is given by the Szekeres-Lazard
theorem \cite[Theorems~4.5.9 and 4.5.13]{AdamsLoustaunau:1994}.

\begin{theorem} \label{ThmSzekeresLazard} (Szekeres-Lazard)
  Let $\E$ be a PID and let $I$ be a nonzero ideal of $\E[x]$. Then $I$ has a
  set of generators $g_0,\ldots,g_m$ of the form
  \begin{align*}
    g_0 &= a_1 a_2 \ldots a_m g, \\
    a_k g_k &= x g_{k-1} + \sum_{i=0}^{k-1} b_{ki} g_i \qquad (1\le k\le m),
  \end{align*}
  for certain $a_k$ and $b_{ki}$ in $\E$, with all $a_k\ne 0$, and with $g$
  equal to $\gcd(I)$.
\end{theorem}

One notes that $\gcd(I)$ is well-defined, because $\E[x]$ is Noetherian (cf.\
\cite[Theorem 1.1.3]{AdamsLoustaunau:1994}) and a factorial ring. Those $g_k$
for which $a_k$ is a unit in $\E$ are not actually needed to generate $I$, and
removing them makes the strong Gr\"obner basis \emph{minimal}.

\begin{example}
  Taking $\E=\ZZ$, a minimal strong Gr\"obner basis for the ideal
  $(x^2+3x+4, 4x^2+3x+1)$ is given by $(48, 3x+69, x^2+3x+4)$. Indeed,
  \begin{align*}
    48 &= (16x^2 + 51x + 68)(x^2+3x+4) - (64x^2 + 60x + 5) (x^2+3x+4), \\
    3x+69 &= (4x+5)(x^2+3x+4) + (-16x+16)(x^2+3x+4), \\
  \intertext{whereas}
    4x^2+3x+1 &= 4(x^2+3x+4) - 3(3x+69) + 4\cdot 48.
  \end{align*}
\end{example}

We note the following consequences of this theorem for our setting.

\begin{corollary} \label{CorMonic}
  Let $\E$ be a PID and let $f_1$ and $f_2$ in $\E[x]$ be coprime. Then the
  ideal $(f_1,f_2)$ contains both a nonzero element $c$ of $\E$ and a monic
  polynomial in $\E[x]$.

  If $\E/(c)$ is finite, then $\E[x]/(f_1,f_2)$ is also finite.
\end{corollary}

\begin{proof}
  We take $c=g_0$, and note that $g_m$ is monic. $\E[x]/(f_1,f_2)$ is then a
  quotient of $(\E/(c))[x]/(g_m)$, which is finite.
\end{proof}

If $\E$ is in fact Euclidean, we can obtain a strong Gr\"obner basis of
$(f_1,f_2)$ by bringing the transpose of the Sylvester matrix of $f_1$ and
$f_2$ into Hermite Normal Form \cite[Theorem 4]{Lazard:1985}.

Conversely, and for $\E$ any PID, we have the following characterisation of the
\emph{resultant} of $f_1$ and $f_2$, in terms of the strong Gr\"obner basis.
Here $\E^*$ is the unit group of $\E$.

\begin{lemma} (Lazard \cite{Lazard:1985}, Myerson \cite{Myerson:1983})
  \label{LemRes}
  Assume the notation of the Theorem, and suppose that the leading coefficients
  of $f_1$ and $f_2$ are coprime. Then
  \begin{equation} \label{EqRes}
    \Res(f_1,f_2) = \prod_{k=1}^m a_k^k.
  \end{equation}
  In particular, under these assumptions, $(f_1,f_2)=(1)$ if and only if
  $\Res(f_1,f_2)\in \E^*$.
\end{lemma}

The second assertion is proved separately in \cite{Myerson:1983}, where we also
find an example of the difficulties of deciding whether $(f_1,f_2)=(1)$ when
the leading coefficients generate a nontrivial ideal. In fact, consider $f_1=
2x+1$ and $f_2=2x+(1+2^e)$ in $\ZZ[x]$, for some $e\ge 0$. Then $I=(f_1,f_2)$
contains the polynomials
$$
  f_2-f_1 = 2^e, \quad 2^{e-1} f_1 - 2^e x = 2^{e-1}, \ldots,
  f_1 - 2x = 1,
$$
so $I$ is trivial. The actual polynomials $u$ and $v$ with minimal degree such
that $uf_1+vf_2=1$ have degree $e$. On the other hand, $\Res(f_1,f_2)= 2^2
(-\frac{1}{2} - \frac{1+2^e}{2}) = 2^{e+1}$.

For the special case $\E=\ZZ$, we have the following.

\begin{lemma} \label{LemRes2}
  Let $f_1$ and $f_2\in\ZZ[x]$ be coprime, with coprime leading coefficients.
  Then the cardinality of $\ZZ[x]/(f_1,f_2)$ is $|\Res(f_1,f_2)|$. If,
  moreover, we have $f_1=x-a$ for some $a\in\ZZ$, then $\Res(f_1,f_2)=f_2(a)$,
  and $\ZZ[x]/(f_1,f_2) \cong \ZZ/(f_2(a))$.
\end{lemma}

\begin{proof}
  The first assertion follows from \eqref{EqRes}, because (as additive groups)
  $$
    \ZZ[x]/(f_1,f_2) \cong \oplus_{k=1}^m \left( \ZZ/a_1\ldots a_k \ZZ \right)
  $$
  by Theorem \ref{ThmSzekeresLazard}.

  Now let $f_1$ be linear, $f_1=x-a$. We have $\Res(f_1,f_2)=f_2(a)$ by the
  definition of resultants (cf.\ \cite[Section IV.8]{Lang:2002:Algebra3}).
  Furthermore, clearly $\ZZ[x]/(f_1)\cong \ZZ$, and the image of $f_2$ inside
  this ring is represented by $f_2(a)$.
\end{proof}

\end{subsection}

\end{section}

\begin{section}{Merging number systems}

\begin{subsection}{Generalities}
We start with two general lemmas, as well as the useful concept of a zero
expansion.

\begin{lemma}  \label{LemProj}
  Let $(V,\phi,\D)$ be a digit system and suppose that we have a
  commutative diagram $\mbox{
    \xymatrix{
      V \ar[r]^\phi \ar@{->>}[d]^\pi & V \ar@{->>}[d]^\pi \\
      W \ar[r]^\psi              & W
      }
  }$,
  where $\pi$ is surjective. Then $(W,\psi,\pi(\D))$ is also a digit system.

  If $(V,\phi,\D)$ is a number system, then so is $(W,\psi,\pi(\D))$.
\end{lemma}

\begin{proof}
  If $d\in\D$ represents $a\in V$ modulo $\phi(V)$, then $\pi(d)$ represents
  $\pi(a)$ modulo $\psi(W)$. Now the result follows by the surjectivity of
  $\pi$.
\end{proof}

We apply the Lemma to the case of polynomial digit systems over a PID $\E$: if
we have a nontrivial factorisation $f=f_1f_2$ in $\E[x]$, and if we set
$V=\E[x]/(f_1f_2)$ and $W=\E[x]/(f_1)$, then obviously the projection
(reduction modulo $f_1$) is surjective, and commutes with multiplication by
$X$. One notes that usually the resulting digit set $\pi(\D)$ will be
redundant.

If we apply the Lemma to both factors $f_1$ and $f_2$ simultaneously, and take
the direct product of the resulting digit systems, we obtain a map $\psi$ into
the direct product, and the outcome is what interests us in this paper.

\begin{example}
Let $\E=\ZZ$, let $f_1=x+2$ and $f_2=x+3$, and let $\N=\{0,1,\ldots,5\}$. It is
known \cite{Gilbert:1981} that $f=f_1f_2 = x^2+5x+6$ is a \emph{CNS
polynomial}, which means that $(\ZZ[x]/(f),X,\{0,\ldots,5\})$ is a number
system (``CNS'' stands for \emph{canonical number system}). Clearly
$\ZZ[x]/(f_1)\cong \ZZ[x]/(f_2)\cong \ZZ$, so the image of $\psi$ gives us two
digit systems in $\ZZ$, namely $(\ZZ,-2,\N)$ and $(\ZZ,-3,\N)$. Both of them
have $6$ digits and are hence redundant; the reason is that the digits
$\{0,\ldots,5\}$, which are pairwise incongruent modulo $x^2+5x+6$, are still
incongruent modulo $x+2$ and $x+3$. However, consider $\N'=\{ 0,X+3, -X-2, 1,
-2X-4,-X-1 \}$. It can be verified that this set is also a valid digit set for
$\ZZ[x]/(x^2+5x+6)$, assuring the Finite Expansion property, and if we apply
the map $\psi$ to it, we find that the image is 
$$
  \{ (0,0),(1,0),(0,1),(1,1),(0,2),(1,2) \}.
$$
It follows that the two resulting digit sets are $\{0,1\}$ and $\{0,1,2\}$,
which are irredundant and in fact well known.
\end{example}

\begin{definition} \label{DefZeroExpansion}
  Let $(V,\phi,\D)$ be a digit system. A \emph{zero expansion} of
  $(V,\phi,\D)$ is a sequence $(d_0,d_1,\ldots,d_\ell)$, with $\ell\ge 0$, of
  digits such that 
  $$
    \sum_{i=0}^{\ell} \phi^i(d_i) = 0.
  $$
  The length of a shortest zero expansion (if one exists) is called the
  \emph{zero expansion length} of the digit system.
\end{definition}

It is easy to prove that when $\D$ is irredundant and $(V,\phi,\D)$ has any
zero expansions at all, then there is a unique \emph{shortest zero expansion},
and all other zero expansions arise as concatenations of copies of the
shortest one. If $0\in\D$, then $(0)$ is obviously a zero expansion.

If a digit system has the Finite Expansion property, then a zero expansion
always exists \cite[Lemma 2.11]{ScheicherSurerThuswaldnerVanDeWoestijne:11}. 
We now show that the lengths of the zero expansions is the only obstruction for
the direct product of number systems to be itself a number system.

\begin{lemma} Let $(V_1,\phi_1,\D_1)$ and $(V_2,\phi_2,\D_2)$ be digit systems.
Then the direct product $(V_1\times V_2,\phi_1\times\phi_2, \D_1 \times \D_2)$
is a number system if and only if
  \begin{romanlist}
    \item 
      $(V_1,\phi_1,\D_1)$ and $(V_2,\phi_2,\D_2)$ are number systems;
    \item 
      we have $\gcd(L_1,L_2)=1$, where $L_i$ is the zero expansion length of
      $(V_i,\phi_i,\D_i)$.
  \end{romanlist}
\end{lemma}

\begin{proof}
  The ``if''-direction is clear, using Lemma \ref{LemProj} if desired.

  Now suppose we have elements $a_i\in V_i$ that have expansions
  $$
    a_i = \sum_{j=0}^{\ell_i} \phi_j^i(d_{ij}) \qquad (d_{ij}\in \D_j; \;
      i=1,2).
  $$
  If we try to put them together in the direct product, to form an expansion
  of the pair $(a_1,a_2)$, we will need the lengths $\ell_1$ and $\ell_2$ to
  be equal. The only way to achieve this is by padding with the shortest
  zero expansion of the number system, as this does not change the value of
  the expansion. As the expansions for the $a_i$ can be chosen independently,
  we need the equation $\ell_1 + u_1L_1 = \ell_2 + u_2L_2$ to be solvable in
  integers $u_1$ and $u_2$ for any given $\ell_1$ and $\ell_2$. Clearly, this
  is equivalent to $\gcd(L_1,L_2)=1$.
\end{proof}

\end{subsection}

\begin{subsection}{Pulling back}

After having considered projections of number systems in the last section, we
now go the other way. Our assumptions are as follows. For the rest of
this section, suppose we are given two \emph{irredundant} digit systems
$(R_i,X,\N_i)$, for $i=1,2$, where $R_i=\E[x]/(f_i)$, and where
$\gcd(f_1,f_2)=1$. We also define $R_{12}=\E[x]/(f_1,f_2)$, and we suppose that
$R_{12}$ is a finite ring (in terms of arithmetic geometry, this implies at
least that we have a complete intersection).

The first question is if we can construct a valid digit set for
$$
  R =_{\rm def} \E[x]/(f_1f_2)
$$
by inverting the map $\psi$ from \eqref{EqPsi} --- in other words, by applying
the Chinese Remainder Theorem, where we recall that $\psi$ is injective.
Unfortunately, if we try to compute $\psi^{-1}(a)$ for some $a=(a_1,a_2)\in
R_1\times R_2$, it turns out in many cases that the inverse image has
nonintegral coefficients; in other words, in general we can only find such an
$a$ in $K[x]/(f_1f_2)$, with $K$ the quotient field of $\E$.  For example, over
$\ZZ$, if $f_1=x+5$ and $f_2=x+7$, and $(a_1,a_2)=(0,1)$, we find $a=-\frac12 y
- \frac52$, and as the CRT asserts that $a$ is unique in $\QQ[x]/(f_1f_2)$,
there is no hope of finding a representative with integral coefficients. This is
exactly the problem of \emph{integral interpolability} that is addressed in
Corollary \ref{CorInterp} above, and in \cite{Petho:2006}. Thus,
we will have to investigate the conditions that ensure the existence of
integral representatives for the new ``composite'' digits.

The answer to this first question already yields several restrictions on the
digit sets.

\begin{lemma} \label{LemDigitsCongruent}
  The following are equivalent:
  \begin{romanlist}
    \item $(d_1,d_2)\in \psi(R)$ for all $d_1\in\N_1$ and $d_2\in\N_2$;
    \item there exists some $\tilde{d}\in\E[x]$ such that $d \equiv \tilde{d}
           \pmod{(f_1,f_2)}$ for all $d\in\N_1 \cup \N_2$.
  \end{romanlist}
\end{lemma}

\begin{proof}
  By Corollary \ref{CorCRT}, (i) implies that $d_1\equiv d_2 \pmod{(f_1,f_2)}$
  for all $d_1\in\N_1$ and $d_2\in\N_2$. Because $d_1$ and $d_2$ are
  independently chosen, it follows that \emph{all digits} are pairwise
  congruent to each other modulo $(f_1,f_2)$, which is (ii). The converse is
  easy by Corollary \ref{CorCRT}.
\end{proof}

Property (ii) of the Lemma could of course interfere with the fact that we
want the $d_i$ to represent all classes of $R_i/(X)$.

\begin{lemma} \label{LemInterference}
  Let $(R_i,X,\N_i)$ be digit systems, for $i=1,2$. If $\N_1$ and $\N_2$
  satisfy the conditions of Lemma \ref{LemDigitsCongruent}, then $f_1(0)$ and
  $f_2(0)$ are coprime.
\end{lemma}

\begin{proof}
  The fact that $\N_i$ represents $\E[x]/(f_i,x)$ means that the constant
  coefficients of the $d\in\N_i$ represent $\E/(f_i(0))$. On the other hand, by
  assumption there is some $\tilde{d}\in\E[x]$ such that
  $d-\tilde{d}\in(f_1,f_2)$ for all $d$; in particular, the residue class of
  $d(0)$ modulo the $\E$-ideal $(f_1(0),f_2(0))$ is constant. Because this
  ideal is generated by $\gcd(f_1(0), f_2(0))$, we see that the gcd is a unit.
\end{proof}

As an example of the last Lemma, consider $f_1=x+2$ and $f_2=x-2$. To satisfy
(ii) of Lemma \ref{LemDigitsCongruent}, we need all digits to be congruent
modulo $4$, as $R_{12}\cong \ZZ/(4)$; but as we need the constant coefficients
of the digits to be both odd and even, this is clearly impossible.

The next question is whether the conditions of the Lemma suffice to transfer
the Finite Expansion property modulo both factors (or just the Periodic
Representation property) to $R$ via $\psi^{-1}$. To settle this question, we
need some definitions and an auxiliary result, which is interesting by itself.

\begin{definition} \label{DefS}
  For $d\in R_{12}$, define the sequence $(s_i(d))_{i\ge 0} \subseteq R_{12}$
  by $s_i(d)= d \sum_{j=0}^i X^j$. Because $R_{12}$ is finite, the sequence
  $(s_i(d))$ is periodic; we let $S(d)$ be the period length.
\end{definition}

\begin{lemma} \label{LemLengthCongruent}
  Assume the conditions of Lemma \ref{LemDigitsCongruent}, and assume that
  the digit systems $(R_i,X,\N_i)$, for $i=1,2$, have Finite Expansions. For
  $i=1,2$, let $a_i\in R_i$ have an expansion $a_i= \sum_{j=0}^{\ell_i} d_{ij}
  X^j$, where $d_{ij}\in \N_i$. Then $a_1\equiv a_2\pmod{(f_1,f_2)}$ if and
  only if $\ell_1 \equiv \ell_2 \pmod{S(\tilde{d})}$, with $\tilde{d}$ the
  common image of all digits in $R_{12}$.
\end{lemma}

\begin{proof}
  Let $\tilde{d}$ be the common congruence class modulo $(f_1,f_2)$ of the
  digits. We have then $ a_i \equiv \sum_{j=0}^{\ell_i} \tilde{d} X^j =
  s_{\ell_i}(\tilde{d})$, for $i=1,2$. Now if $a_1\equiv a_2\pmod{(f_1,f_2)}$,
  we have $s_{\ell_1}(\tilde{d})= s_{\ell_2}(\tilde{d})$, so that
  $\ell_2-\ell_1$ is divisible by the period length $S(\tilde{d})$. Conversely,
  if $\ell_1\equiv\ell_2\pmod{S(\tilde{d})}$, then by definition
  $s_{\ell_1}(\tilde{d})=s_{\ell_2}(\tilde{d})$, and it follows that $a_1\equiv
  a_2\pmod{(f_1,f_2)}$.
\end{proof}

This brings us to our main result, which employs the concept of zero
expansions (see Definition \ref{DefZeroExpansion}).

\begin{theorem} \label{ThmCRTNumberSystem}
  Let $\E$ be a PID. For $i=1,2$, let $R_i=\E[x]/(f_i)$ and let $\N_i\subseteq
  R_i$ be a finite set. Assume $\gcd(f_1,f_2)=1$, put $R=\E[x]/(f_1f_2)$, and
  put
  $$
    \psi : R \rightarrow R_1 \times R_2 :
         a \mapsto (a \bmod{f_1}, a\bmod{f_2}).
  $$
  Then $(R,X,\psi^{-1}(\N_1\times \N_2))$ is an irredundant digit system if and
  only if
  \begin{romanlist}
    \item \label{crti}
      $(R_i,X,\N_i)$ is an irredundant digit system for $i=1,2$;
    \item \label{crtii}
      there exists some $\tilde{d}\in\E[x]$ such that $d\equiv
      \tilde{d}\pmod{(f_1,f_2)}$ for all $d\in \N_1\cup \N_2$.
  \end{romanlist}
  Assume \eqref{crti} and \eqref{crtii}, and assume furthermore that
  $R_{12}=\E[x]/(f_1,f_2)$ is finite. Then $(R,X,\psi^{-1}(\N_1\times \N_2))$
  has the Finite Expansion property, with zero expansion length
  $\gcd(L_1,L_2)$, if and only if
  \begin{romanlist}
    \setcounter{enumi}{2}%
    \item \label{crtiii}
      $(R_i,X,\N_i)$ has the Finite Expansion property, with zero expansion
      length $L_i$, for $i=1,2$;
    \item \label{crtiv}
      $\gcd(L_1,L_2)=S(\tilde{d})$, where $S(\tilde{d})$ is the period length
      of Definition \ref{DefS}.
  \end{romanlist}
\end{theorem}

Note that by Corollary \ref{CorMonic}, the ring $R_{12}$ is finite if and
only if we have $\E/(c)$ finite for some $c\in(f_1,f_2)\cap \E$. For rings $\E$
such as $\ZZ$ and $\FF[x]$ for a finite field $\FF$, where the quotient by
every nonzero ideal is finite, the finitude of $R_{12}$ therefore follows from
the assumption $\gcd(f_1,f_2)=1$.

\begin{proof}
  By Corollary \ref{CorCRT}, we have $R \cong R_1 \times_{R_{12}} R_2$.
  By \eqref{crtii} and Lemma \ref{LemDigitsCongruent}, we know that $\N_1\times
  \N_2\subseteq \psi(R)$, so that we can define $\N=\psi^{-1}(\N_1 \times
  \N_2)$.

  For the question whether $\N$ is a system of representatives modulo $X$ in
  $R$, we reduce everything modulo $X$ and are reduced to the same question for
  $$
    \E/(f_1(0)f_2(0)) \cong \E/(f_1(0)) \times \E/(f_2(0)),
  $$
  where the isomorphism holds by the usual Chinese Remainder Theorem. This
  works here because $\E$ is a PID, and because $f_1(0)$ and $f_2(0)$ are
  coprime by Lemma \ref{LemInterference}.

  Conversely, if $\N$ exactly represents $R$ modulo $X$, then \eqref{crtii}
  follows because $\N$ is defined at all; the image of $\N$ under reduction
  modulo $f_i$ is $\N_i$, so $\N_i$ represents $R_i$ modulo $X$; and in fact,
  the simultaneous representation in $R_1\times R_2$ is exact by cardinality
  considerations.

  We now turn to the more interesting second assertion, starting with the
  ``if''-part. Let $a\in R$; we must show that $a$ has a finite expansion on
  the basis $X$ with digits in $\N$.

  Let $(a_1,a_2)=\psi(a)$. For $i=1,2$, by \eqref{crtiii}, we have expansions
  $a_i= \sum_{j=0}^{\ell_i} d_{ij} X^j$, where $d_{ij}\in \N_i$. By Lemma
  \ref{LemLengthCongruent}, we have $\ell_1\equiv\ell_2\pmod{S(\tilde{d})}$; we
  claim that we may assume $\ell_1=\ell_2$. Indeed, because
  $\gcd(L_1,L_2)=S(\tilde{d})$ by \eqref{crtiv}, we can find nonnegative
  integers $u_1$ and $u_2$ such that
  $$
    \ell_1 + u_1 L_1 = \ell_2 + u_2 L_2.
  $$
  Thus, after padding the expansion of $a_i$ with $u_i$ times the shortest zero
  expansion of $(R_i,X,\N_i)$, which does not change the value of the
  expansion, we may assume that $\ell_1 = \ell_2$. It follows that
  $$
    a = \sum_{j=0}^{\ell_1} \psi^{-1}(d_{1j}, d_{2j}) X^j,
  $$
  which is the desired expansion with digits in $\N$. Taking $a_1=a_2=0$ shows
  that the zero expansion length of $(R,X,\N)$ is $\gcd(L_1,L_2)$.

  We prove the ``only if''-part. The fact that \eqref{crtiii} follows from the
  Finite Expansion property for $(R,X,\N)$ was already shown at the
  beginning of the section.

  To prove \eqref{crtiv}, we take $a_i=\sum_{j=0}^{\ell_i} d_{ij} X^j$, for
  $i=1,2$, where $\ell_1$ and $\ell_2$ and the digits $d_{ij}$ are chosen
  arbitrarily, such that $\ell_1\equiv \ell_2\pmod{S(\tilde{d})}$. By Lemma
  \ref{LemLengthCongruent}, we have $a_1\equiv a_2\pmod{(f_1,f_2)}$. Expanding
  $\psi^{-1}(a_1,a_2)$ in $R$ and again applying $\psi$, we find expansions for
  $a_1$ and $a_2$, on their respective digit sets, of \emph{equal lengths}.
  Because $\ell_1$ and $\ell_2$ were arbitrary, it follows in particular that
  the lengths of the shortest zero expansions satisfy $(L_1,L_2)=S(\tilde{d})$,
  as desired.
\end{proof}

\begin{example}
  We let $f_1=x+2$, $f_2=x+3$, and $R_i=\ZZ[x]/(f_i)$, using the classical
  digits. Thus, in fact we have $(\ZZ,-2,\{0,1\})$ and $(\ZZ,-3,\{0,1,2\})$,
  which obviously have the Finite Expansion property, as ``starting'' digit
  systems. We have $\Res(f_1,f_2)=1$, so by Lemma \ref{LemRes2}, $R_{12}$ is
  the zero ring, we can take $\tilde{d}=0$, and we find $S(\tilde{d})=1$. We
  also have $L_1=L_2=1$, so that all assumptions are satisfied. By the Theorem,
  we find that
  $$
    \{ 0,X+3, -X-2, 1, -2X-4,-X-1 \} =
    \psi^{-1}\left(\{ (0,0),(1,0),(0,1),(1,1),(0,2),(1,2) \}\right)
  $$
  is a valid digit set for $\ZZ[x]/(x^2+5x+6)$. Indeed, the product $R_1\times
  R_2$ is isomorphic to the entire ring $R=\ZZ[x]/(f_1f_2)$, and we have
  $f_1f_2 = x^2+5x+6$.
\end{example}

\end{subsection}

\begin{subsection}{Necessary conditions}

For the general case, where $R_{12}$ is not necessarily trivial, we assemble a
number of necessary conditions on the ``starting'' number systems in the
following result. One notes in particular that in this case, the common residue
class $\tilde{d}$ for all digits cannot be $0$, because $0$ is not invertible,
unless $R_{12}$ is the zero ring. This implies that neither $\N_1$ nor $\N_2$
may contain $0$, unless $R_{12}=0$.

\begin{theorem} \label{ThmMaximalPeriod}
  Assume that the digit systems $(R_i,X,\N_i)$, for $i=1,2$, satisfy conditions
  \eqref{crti}--\eqref{crtiv} of the Theorem, and that $R_{12}$ is finite. Let
  $\tilde{d} $ be the common congruence class modulo $(f_1,f_2)$ of the digits.
  Then $\tilde{d}$ is invertible in $R_{12}$, and we have
  $S(\tilde{d})=|R_{12}|$.

  Assume in addition that $\E=\ZZ$, and that the leading coefficients of $f_1$
  and $f_2$ are coprime. Then $S(\tilde{d})=|\Res(f_1,f_2)|$, and if $f_1=X-a$
  for some $a\in\ZZ$, we have $a \equiv 1\pmod{p}$ for all primes $p$ dividing
  $f_2(a)$, and $a \equiv 1\pmod{4}$ if $4$ divides $f_2(a)$.
\end{theorem}

\begin{proof}
  The map $R_1\times_{R_{12}} R_2\rightarrow R_{12}$ sending $(a_1,a_2)$ to
  $a_1\bmod{(f_1,f_2)}$ is surjective, so that every element of $R_{12}$ has an
  expansion of the form $\sum_{j=0}^\ell \tilde{d} X^j = s_\ell(\tilde{d}) =
  \tilde{d} s_\ell(1)$. It follows that the set $\{ \tilde{d} s_i(1) : i \ge 0
  \}$ covers all elements of $R_{12}$. Hence, $\tilde{d}$ must be a unit, and
  the period $S(1)$ of the sequence $(s_i(1))$ must be equal to the cardinality
  of $R_{12}$, which (in the case of $\E=\ZZ$ and coprime leading coefficients)
  is $|\Res(f_1,f_2)|$ by Lemma \ref{LemRes2}.

  Now assume that $f_1=X-a$, with $\E=\ZZ$. We then have $R_{12} \cong
  \ZZ/(f_2(a))$, and the class of $X$ in $R_{12}$ is represented by $a$.
  Consider the sequence $(s_i(1))$: we have $s_0(1)=1$, and $s_{i+1}(1) =
  as_i(1) + 1 \pmod{S(1)}$. It follows that $(s_i(1))$ is a \emph{linear
  congruential sequence}. By Knuth's theorem \cite[Theorem
  3.2.1.2A]{Knuth:1998}, its period $S(1)$ is maximal if and only if $a\equiv
  1\pmod{p}$ for all primes $p|S(1)$, and $a\equiv 1\pmod{4}$ if $4|S(1)$.
\end{proof}

\begin{example}
  We will take $f_1=X+3$ and $f_2=X+5$, so that $\Res(f_1,f_2)=2$ and
  $R_{12}\cong \ZZ/2\ZZ$, the field with $2$ elements. By Theorem
  \ref{ThmMaximalPeriod}, the common residue class for all digits must be $1$;
  it follows that all digits must be odd. By Corollary 2.15 of
  \cite{VanDeWoestijne:09}, we find that we may take
  $$
    (\ZZ,-3,\{-3,1,-1\}) \quad \text{and} \quad (\ZZ,-5,\{-5,1,-3,3,-1\})
  $$
  as starting number systems. We verify that the sequence $s_i(1)$ covers both
  elements of $R_{12}$; also, we have $L_1=L_2=2$, so that all assumptions are
  satisfied. Now Theorem \ref{ThmCRTNumberSystem} tells us that
  \begin{align*}
    \{ & X, \, 1, \, X + 2, \, -3X - 12, \, X + 4, \, 2X + 5, \, -2X - 9, \,
    2X + 7, \, -2X - 7, \, -X - 6, \\
       & 3X + 10, \, -X - 4, \, -3, \, -X - 2, \, -1 \}
  \end{align*}
  is a valid digit set for $f_1f_2=X^2+8X+15$. As an example, we have
  $2X + 7 \equiv 1\pmod{X+3}$ and $\equiv -3\pmod{X+5}$, and as a random
  example of an expansion, we have
  \begin{multline*}
    37X - 55 \equiv  (2X + 5) \cdot 1 + (2X + 7) \cdot X + X \cdot X^2 +  \\
                ( -X - 4) \cdot X^3 + (3X + 10) \cdot X^4 + (X + 4)\cdot X^5
    \pmod{X^2+8X+15}.
  \end{multline*}
  Unfortunately, one notes that Theorem \ref{ThmMaximalPeriod} poses several
  conditions on the defining polynomials of the number systems themselves,
  conditions which are largely independent of the chosen digit sets. For
  example, let $f_1=X+4$ and $f_2=X+7$; we have $\Res(f_1,f_2) =3$ and $R_{12}
  \cong \ZZ/3\ZZ$. In $R_{12}$, we have $X=2$, which does not satisfy the
  conditions of Theorem \ref{ThmMaximalPeriod}. It follows that there exist
  \emph{no digit sets} $\N_1$ and $\N_2$ such that $(\ZZ,-4,\N_1)$ and
  $(\ZZ,-7,\N_2)$ satisfy conditions \eqref{crti}--\eqref{crtiv} of Theorem
  \ref{ThmCRTNumberSystem}.
\end{example}


\end{subsection}

\begin{subsection}{Redundant digit sets}
One final remark. In all the above, we have assumed the $\N_i$ to be
irredundant. The case where the digit sets $\N_i$ are allowed to be redundant
is much more difficult to control, as is in fact apparent at every step. The
biggest problem is that expansions are no longer unique. In general, if we
start from any digit set for $f=f_1f_2$ and project to one factor, the result
will be redundant --- for example, if we take the classical digits
$\{0,1,\ldots,|f(0)|-1\}$, this is always the case. This means that it is
still hard to use the Theorem to say anything about the CNS property in
relation to factorisation of polynomials.
\end{subsection}

\end{section}

\begin{section}{Simultaneous number systems} \label{SecIndle}

An interesting generalisation of the ordinary number systems in $\ZZ$ is
obtained if we try to expand the same integer on several bases at once, while
using the same digit sequence for all bases. Such \emph{simultaneous number
systems} were considered in \cite{IndlekoferKataiRacsko:1992} and
\cite{Petho:2006}; here, we generalise them in several respects, and we reprove
and extend the main result of \cite{IndlekoferKataiRacsko:1992} on this topic,
as an application of the theory developed in the last section.

\begin{example}
To see what simultaneous number systems are about, consider the double
expansion
$$
  100 = (1 5 3 3 4 4)_{(-3,-4)};
$$
it means that the digit sequence $(1,5,3,3,4,4)$, starting with the most
significant digit, yields $100$ both in base $-3$ and in base $-4$
simultaneously! By looking at the least significant digits, one sees that the
digits used here must cover both $\ZZ$ modulo $3$ and modulo $4$; in other
words, by the Chinese Remainder Theorem, the digits must cover $\ZZ$ modulo
$12$.

We illustrate the derivation of such an expansion, using digits
$\{0,\ldots,11\}$ and bases $N_1=-3$ and $N_2=-4$. For any pair
$(a_1,a_2)\in\ZZ^2$, the notation $(a_1,a_2)\stackrel{d}{\rightarrow}
(b_1,b_2)$ means that $(a_i-d)/N_i=b_i$, for $i=1,2$, where the divisions are
exact in $\ZZ$. Taking the example $a=100$, this gives
\begin{multline*}
  (100,100) \stackrel{4}{\rightarrow} \left( \frac{100-4}{-3}, \frac{100-4}{-4}
  \right) = (-32,-24) \stackrel{4}{\rightarrow} (12,7)
  \stackrel{3}{\rightarrow} (-3,-1) \\
   \stackrel{3}{\rightarrow} (2,1) \stackrel{5}{\rightarrow} (1,1)
  \stackrel{1}{\rightarrow} (0,0),
\end{multline*}
and we obtain the expansion given above by reading off the digits in reverse
order. In each case, the digits are found using the Chinese Remainder Theorem.
\end{example}

To formalise these observations, the following definition was proposed by
by Indlekofer, K\'atai, and Racsk{\'o} \cite[Section
4]{IndlekoferKataiRacsko:1992}. 

\begin{definition} \label{DefSim1} (First version)
  Given an integer $k\ge 1$, and pairwise coprime integers $N_1,\ldots,N_k$
  unequal to $0$, let $V=\ZZ^k$, and define
  \begin{align*}
    &\phi: V\rightarrow V: (a_1,\ldots,a_k)\mapsto (N_1a_1,\ldots,N_ka_k); \\
    &\N = \{ (c,c,\ldots,c) \mid c =0,1,\ldots,|N_1\cdots N_k|-1 \}
    \subseteq V.
  \end{align*}
  Then $(V,\phi,\N)$ is called the \emph{simultaneous digit system} defined by
  the $N_i$. If it has the Finite Expansion property, we call it the
  \emph{simultaneous number system} defined by the $N_i$.
\end{definition}

We will give the main result in the following more general setting.

\begin{definition} \label{DefSim2} (Second version)
  Let $\E$ be a PID, let $k\ge 1$ be an integer, let $f_1,\ldots,f_k$ in
  $\E[x]$ be pairwise coprime such that also $f_1(0),\ldots,f_k(0)$ are
  pairwise coprime in $\E$, let $\R\subseteq \E[x]$ be any set of polynomials
  such that their constant coefficients form a complete system of
  representatives of $\E$ modulo $f_1(0)\cdots f_k(0)$, and define
  \begin{align*}
    &V=\E[x]/(f_1) \times \cdots \times \E[x]/(f_k); \\
    &\phi: V\rightarrow V: (a_1,\ldots,a_k)\mapsto (Xa_1,\ldots,Xa_k); \\
    &\N = \{ (c,c\ldots,c) \mid c \in\R \}.
  \end{align*}
  Then $(V,\phi,\N)$ is called the \emph{simultaneous digit system} defined by
  the $f_i$ and $\R$. If it has the Finite Expansion Property, we call it the
  \emph{simultaneous number system} defined by the $f_i$ and $\R$.
\end{definition}

To recover the systems from the first version of the definition,
one simply takes $\E=\ZZ$, $f_i=x-N_i$, and $\R=\{0,1,\ldots,|N_1\cdots
N_k|-1\}$. Note that in the general setting, the digits need not be elements of
$\E$; however, their representative properties depend only on their constant
coefficients, as is easily seen.

The main result of this section is as follows.

\begin{theorem} \label{ThmSim}
  Let $f_1,\ldots,f_k\in\E[x]$ and $\R\subseteq\E[x]$ define a simultaneous
  digit system $(V,\phi,\N)$. Then $(V,\phi,\N)$ has the Finite Expansion
  Property if and only if
  \begin{romanlist}
    \item
      we have the equality of ideals $(f_i,f_j)=(1)\subseteq\E[x]$ for all
      $i\ne j$; 
    \item
      the digit system $(\E[x]/(f_1f_2\cdots f_k),X,\R)$ also has the Finite
      Expansion Property.
  \end{romanlist}
\end{theorem}

\begin{proof}
  Let $R_i=\E[x]/(f_i)$; then $V=R_1\times \cdots \times R_k$. Expansions in
  $(V,\phi,\N)$ are of the form
  $$
    (v_1,\ldots,v_k) = \sum_i (c_i,c_i,\ldots,c_i) X^i,
  $$
  where the $i$th component is taken modulo $f_i$. By the form of the basis and
  of the digits, all such expansions are contained in the $\E[x]$-submodule $W$
  of $V$ consisting of all vectors of the form $(a \bmod{f_1},\ldots,
  a\bmod{f_k})$ for some $a\in\E[x]$. By Corollary \ref{Cor111}, the module $W$
  is isomorphic to $R=\E[x]/(f_1\cdots f_k)$.

  It follows that all elements of $V$ have a finite expansion in the digit
  system $(V,\phi,\N)$ if and only if $V=W$ and the digit system $(R,X,\R)$
  has the Finite Expansion Property. By Corollary \ref{Cor111}, we have $V=W$
  if and only if all ideals $(f_i,f_j)$, when $i\ne j$, are trivial.
\end{proof}

The Theorem should be compared to Theorem 3 of \cite{Petho:2006}, which says
that, assuming that $(R,X,\R)$ has the Finite Expansion property, any
particular vector $(a_1,\ldots,a_k)\in V$ has a finite expansion if and only if
it is in $W$. In fact, it follows from Corollary \ref{Cor111} that
interpolability by a polynomial with integral coefficients, which is the
property used in \cite{Petho:2006}, is equivalent to being in $W$ --- this is
Corollary \ref{CorInterp} above.

We now show that ``classical'' simultaneous number systems in $\ZZ^k$, in the
sense of Definition \ref{DefSim1}, exist only for $k=1$ and $k=2$. The case
$k=2$ was already given in \cite{IndlekoferKataiRacsko:1992}, but our proof is
much shorter.  Recall that a monic $f\in\ZZ[x]$ is a \emph{CNS polynomial} if
and only if the digit system $(\ZZ[x]/(f),X,\{0,1,\ldots,|f(0)|-1\})$ has the
Finite Expansion Property.

\begin{corollary} \label{CorSim}
  For $i=1,\ldots,k$, let $N_i\in\ZZ$ and $f_i=x-N_i$. Let $\R=\{0,1,\ldots,
  |N_1\cdots N_k|-1 \}$. Then the $f_i$ and $\R$ define a simultaneous number
  system if and only if we have $N_i\le -2$ for all $i$ and either $k=1$, or
  $k=2$ and $|N_1-N_2|=1$.
\end{corollary}

\begin{proof}
  By the Theorem, we need $\prod f_i$ to be a CNS polynomial. In particular,
  all $N_i$ must be $\le -2$, as a CNS polynomial must be expanding and
  cannot have positive real roots.

  Then the case $k=1$ is trivial; we assume $k\ge 2$.

  As for necessity, assume we have Finite Expansions; then the Theorem tells us
  that we have $(f_i,f_j)=(1)$ for all $i\ne j$. By Myerson's Lemma
  \ref{LemRes}, this is equivalent to having $|\Res(f_i,f_j)|=1$ for all
  $i\ne j$. Now $\Res(x-a,x-b)=a-b$ for any $a$ and $b$. One sees easily
  that the only possibility is to have $k=2$ and $|N_1-N_2|=1$.

  The sufficiency follows from the fact that $(x-N_1)(x-N_1+1)$, with
  $N_1\le -2$, is a CNS polynomial by Gilbert's criterion \cite{Gilbert:1981}.
  All conditions of the Theorem are satisfied, and we conclude that we have
  Finite Expansions.
\end{proof}

One notes that Definitions \ref{DefSim1} and \ref{DefSim2} quite restrictive,
in requiring that all elements of $V$ have a finite expansion.  If, for
example, we only want simultaneous expansions \emph{for the elements of $\E$},
then the equality $V=W$ from the proof of the Theorem is unnecessary, because
both $\E$ (represented as $\{ (a\bmod{f_1},\ldots,a\bmod{f_k}) : a \in \E \}$),
the base $X$, and the digits $\N$ are all contained in $W$. By Corollary
\ref{Cor111}, the problem is then translated directly to the question whether
all elements of $\E$ have a finite expansion under the digit system $(R,X,\R)$
(notations as in the proof of the Theorem).

There are open problems here: in case $E=\ZZ$, for example, it is unclear
whether the property that all elements of $\ZZ$ have a finite expansion in the
CNS digit system $(\ZZ[x]/(f),X,\{0,\ldots,|f(0)|-1\})$ is enough to imply that
$f$ is a CNS polynomial --- i.e., that \emph{all} elements of $\ZZ[x]/(f)$ have
a finite expansion. At the moment, I do not know any counterexample.
Algorithmically speaking, this property can be verified for a given $f$ by
applying Brunotte's witness criterion, where the starting set contains only
$\pm 1$, instead of a complete set of semigroup generators for $V$.

Sidestepping this problem for the moment, we see that the question of
classifying simultaneous number systems with \emph{classical digits} is also
related to the problem of determining whether the product of CNS polynomials is
again a CNS polynomial. This is a well-known open problem. It was proved
independently in \cite{Petho:2006} and \cite{Kane:2006} that the product of up
to $4$ linear CNS polynomials is again CNS; but \cite{Kane:2006} also gives an
example of a product of $9$ linear CNS polynomials that is not CNS.

An important difficulty here, however, is that the classical digit sets, when
projected down to a factor of the defining polynomial, will become redundant.
This means that it is possible for a polynomial $fg$ to be CNS, although 
neither $f$ nor $g$ is. I do not know of a concrete example here; it would be
interesting to find such examples.

At least, by \cite{Petho:2006} and \cite{Kane:2006}, we have the following
result.

\begin{theorem}
  Let $N_1,\ldots,N_k\in\ZZ$ be pairwise coprime and less than or equal to
  $-2$. Then all integers have a finite expansion in the simultaneous digit
  system defined by the $N_i$, whenever $(x-N_1)\cdots (x-N_k)$ is a CNS
  polynomial. In particular, if $k\le 4$, the conclusion always holds.
\end{theorem}

\begin{example}
  Let $N_1=-2$, $N_2=-3$, and $N_3=-5$, with digit set $\{0,\ldots,29\}$. We
  consider the conditions for triples $(a_1,a_2,a_3)\in \bigoplus_i
  \ZZ[x]/(x-N_i)$ to be in the ``triple fibred product'' $W$: in fact, we need
  $a_i\equiv a_j\pmod{(x-N_i,x-N_j)}$ for $i\ne j$.
  
  We have $\Res(x+2,x+3)=-1$, hence $(x+2,x+3)=(1)$, so $a_1$ and $a_2$ can be
  independently chosen. But $\Res(x+2,x+5)=-3$ and $\Res(x+3,x+5)=-2$, so we
  need $a_1\equiv a_3\pmod{3}$ and $a_2\equiv a_3\pmod{2}$. We see that for 
  every choice of $a_1$ and $a_2$, the choice of $a_3$ is already determined
  modulo $6$.

  However, $(x+2)(x+3)(x+5)$ is a CNS polynomial by the results mentioned
  before, and we conclude that every integer has a unique simultaneous
  expansion modulo these three bases, with digits in $\{0,\ldots,29\}$.
  Furthermore, every triple $(a_1,a_2,a_3)$ satisfying the above conditions
  is expansible as well. For example, we have the nontrivial cycle
  $$
    (1,1,6) \overset{1}{\rightarrow} (0,0,-1) \overset{24}{\rightarrow}
    (12,8,5) \overset{20}{\rightarrow} (4,4,3) \overset{28}{\rightarrow}
    (12,8,5) \overset{20}{\rightarrow} \ldots,
  $$
  showing that $(1,1,6)$ is not expansible, whereas
  $$
    (1,1,7) \overset{7}{\rightarrow} \left( \frac{1-7}{-2},\frac{1-7}{-3},
    \frac{7-7}{-5}\right) = (3,2,0) \overset{5}{\rightarrow} (1,1,1)
    \overset{1} {\rightarrow} (0,0,0).
  $$
  Note that expansions will not be equal to the usual $30$-ary (or the
  less usual $(-30)$-ary) expansions; for example, the expansion length of
  integers $a$ will be proportional to $\log_2 a$, instead of $\log_{30} a$.
\end{example}

Interesting simultaneous number systems (in the strict sense, where we want
all elements of $V$ to be representable) with more than $2$ components can be
constructed if we allow the defining polynomials to be nonlinear. An infinite
family of quadratic triples that give rise to simultaneous number systems is
given in the next result.

\begin{theorem}
  Let $a\in\ZZ$ with $a\le -7$, let $f_a=(x-a)(x-a-1)-1$, $g_a=f_a+x-a-1$, and
  $h_a=f_a+x-a-2$. Let $\R = \{0,1,\ldots,|f_a(0)g_a(0)h_a(0)|-1\}$.
  Then $f_a$, $g_a$, and $h_a$ are irreducible and coprime, and together with
  $\R$ define a simultaneous number system.
\end{theorem}

\begin{proof}
  Because $h_a=g_a-1$, we clearly have $\Res(g_a,h_a)=1$. Next, using
  properties of the resultant, we compute
  \begin{align*}
    \Res(f_a,g_a) &= \Res(f_a,x-a-1) = f_a(a+1) = -1; \\
    \Res(f_a,h_a) &= \Res(f_a,x-a-2) = f_a(a+2) = 1.
  \end{align*}
  Thus, by Theorem \ref{ThmSim} and Lemma \ref{LemRes}, it is enough to prove
  that $f=f_ag_ah_a$ is a CNS polynomial.

  Now because $f_a(a-1)=f_a(a+2)=1$ and $f_a(a)=f_a(a+1)=-1$, we conclude that
  $f_a$ is irreducible, and the same reasoning works for $g_a$ and $h_a$.
  The same argument shows that all three have their (real) zeros in the
  interval $(a-2,a+2)$. We now assume $a\le -3$, so that $f$ is expanding and
  has only negative real roots. Write $f=\sum_{i=0}^6 c_ix^i$; we have, for
  $|a|\ge 2$ and $i=0,\ldots,6$,
  $$
    \binom{6}{i}(|a|-2)^i < c_i < \binom{6}{i} (|a|+2)^i.
  $$
  This allows us to conclude that $c_{i-1}<c_i$ for $i=1,\ldots,6$, as soon as
  $a\le -20$, so that $f$ is a CNS polynomial by a well-known criterion (for
  example, this follows from Proposition 7 of \cite{Gilbert:1981}). By explicit
  calculation of the coefficients, one obtains the same result for $-19\le a\le
  -7$.
\end{proof}

For $-6\le a\le -3$, the polynomial $f$ obtained in the proof of the Corollary
no longer has strictly increasing coefficients, and only rather heavy
computation can tell us whether it is a CNS polynomial. The case $a=-3$ gives
at once the smallest example of three monic expanding quadratic polynomials
having pairwise resultant $\pm 1$: we have $(f_{-3}, g_{-3},
h_{-3})=(x^2+5x+5,\, x^2+6x+6,\, x^2+6x+7)$.

Unfortunately, the product
$$
  f_{-3}g_{-3}h_{-3}=x^6 + 17x^5 + 114x^4 + 383x^3 + 677x^2 + 600x + 210
$$
is not a CNS polynomial: using Brunotte's witness set criterion
\cite[Lemma~2]{Brunotte:2001}, one computes a witness set of $153807$ elements
that contains a $14$-cycle starting in $x^4 + 16x^3 + 98x^2 + 285x + 392$,
showing that the periodic set of this digit system contains nonzero elements.
However, the corresponding products for $a=-4,-5,-6$ do turn out to be CNS
polynomials, with much smaller witness sets. One notes that $f_{-3}$, $g_{-3}$,
and $h_{-3}$ all have a real root between $-2$ and $-1$, whereas for $a\le -4$,
all roots of the involved polynomials are real and less than $-2$. This
corresponds with the general observation that problems about number systems
become easier when all conjugates of the base are greater than $2$ in modulus.

The same proof shows that the triples $(f_a,\,x-a-1,\,x-a-2)$ of one quadratic
and two linear polynomials, for $a$ negative and large enough, also generate
simultaneous number systems.

One also obtains infinite families of pairwise-resultant-$1$-triples by taking
$g_a=f_a+x-a$ or $f_a+x-a+1$, with $h_a=g_a-1$. However, these choices for
$g_a$ are reducible: one has $f_a+x-a=(x-a-1)(x-a+1)$ and $f_a+x-a+1=(x-a)^2$.
They are expanding, so one can use them all the same for defining digit
systems; to complete the argument, one needs to show that their product is a
CNS polynomial. In fact, if $a$ is negative and large enough, this will ensue
automatically; the proof is the same as above.

Finally, one is led to the question of characterising all sets of integral
polynomials which have pairwise resultant $\pm 1$. If we take $v$ monic
polynomials of degree $d$, and take all non-leading coefficients as variables,
this leads to an algebraic variety in $vd$-dimensional affine space, cut out by
$\binom{v}{2}$ resultant equations. It follows by dimension considerations that
when $v\ge 2d+1$, we cannot expect any solutions, unless the intersection is
incomplete. Experimentally, it is easy to find quadruples of monic cubic
polynomials having pairwise resultant $\pm 1$; there are no monic cubic
quintuples with coefficients in $\{-3,\ldots,3\}$.

A set of irreducible monic nonconstant polynomials having pairwise resultant
$1$ (when pairs are chosen in the given order) is given by
$$
  \{ x-1,\: x,\: x^2-x+1,\: x^3-x+1,\: x^4-x^3+x^2-x+1,\: 
    x^5-2x^3+3x^2-2x+1 \}.
$$
It can be proved using rather extensive geometric computations that this set
is maximal among polynomials of degree at most $5$. Reducibility and resultants
do not change when we substitute $x-a$ for $x$ (with $a\in\ZZ$), so when we
take $a$ negative and large enough, the product of the six shifted polynomials
will have monotonically increasing coefficients (as above) and hence become a
CNS polynomial, and we obtain infinitely many simultaneous number systems.
  
\end{section}

\def\polhk#1{\setbox0=\hbox{#1}{\ooalign{\hidewidth
  \lower1.5ex\hbox{`}\hidewidth\crcr\unhbox0}}} \def\cprime{$'$}
  \hyphenation{ma-the-ma-ti-ques}

\end{document}